\documentclass{amsart}

\usepackage[T1]{fontenc}
\usepackage{tgtermes}
\usepackage{mathptmx}
\usepackage{bussproofs}
\usepackage[scaled=.92]{helvet}
\usepackage{amsthm,amsmath,amssymb}
\usepackage{mathrsfs}
\usepackage[numbers]{natbib}  
\usepackage[colorlinks,citecolor=blue,urlcolor=blue]{hyperref}  

\usepackage{enumerate}

\newtheorem{theorem}{Theorem}[section]
\newtheorem{lemma}[theorem]{Lemma}
\newtheorem{proposition}[theorem]{Proposition}

\theoremstyle{definition}
\newtheorem{definition}[theorem]{Definition}

\theoremstyle{remark}

\newcommand{\myd}[1]{\text{\d{$#1$}}}

  \title{IKT$^\omega$ and \L{}ukasiewicz-models}

  \author[Fjellstad]{Andreas Fjellstad} 

  \address{Department of Philosophy\\
    University of Bergen\\
    Postboks 7805\\
    5020 Bergen\\
    Norway
    }%
      \email{afjellstad@gmail.com}

  \author[Olsen]{Jan-Fredrik Olsen}
  \address{Centre for Mathematical Sciences\\
    Lund University\\
    P.O. Box 118\\
    SE-221 00 Lund\\
    Sweden
    }
    	\email{jan-fredrik.olsen@math.lu.se}

\keywords{non-contractive truth, multiplicative quantifiers, infinitary sequent, soundness, $\omega$-inconsistency, inconsistency, vacuous quantification}

\subjclass[2010]{03B47, 03C90, 03C75}

\begin{document}

\maketitle

\begin{abstract}
In this note, we  show that the first-order logic IK$^\omega$ is sound with regard to the models obtained from continuum-valued \L{}ukasiewicz-models for first-order languages by treating the quantifiers as infinitary strong disjunction/conjunction rather than infinitary weak disjunction/conjunction. Moreover, we show that these models cannot be used to provide a new consistency proof for the theory of truth IKT$^\omega$ obtained by expanding IK$^\omega$ with transparent truth because the models are inconsistent with transparent truth. Finally, we show that whether or not this inconsistency can be reproduced in the sequent calculus for IKT$^\omega$ depends on how vacuous quantification is treated.
\end{abstract}

\section{Introduction}
In Zardini \cite{Zardini2011-ZARTWC}, IKT$^\omega$ is presented as the theory of truth obtained by expanding the logic IK$^\omega$ with a transparent truth predicate. While the propositional fragment of IK$^\omega$ is propositional affine logic, that is, linear logic with weakening, the quantifiers of IK$^\omega$ are generalisations of multiplicative conjunction and disjunction as opposed to generalisations of additive conjunction and disjunction. To capture multiplicative quantifiers, \cite{Zardini2011-ZARTWC} relies on a sequent calculus which is infinitary not only in the sense that it contains rules with infinitely many premises, but also that the sequents may contain infinitely many formulas.

The aim of this paper is to shed some light on the multiplicative quantifiers introduced by \cite{Zardini2011-ZARTWC}. We show first that the sequent calculus for the logic IK$^\omega$ is sound with regard to models obtained from first-order continuum-valued \L{}ukasiewicz-models by treating the quantifiers as infinitary strong disjunction/conjunction rather than infinitary weak disjunction/conjunction. One might thus think that these models can be used to provide a new consistency proof for IKT$^\omega$ to replace the consistency proof presented by \cite{Zardini2011-ZARTWC}, which was found to be erroneous by Fjellstad \cite{Fjellstad2020-FJEANO}. Instead, we proceed to show that not only are these models inconsistent with transparent truth, but we can also reproduce the inconsistency in the sequent calculus for IKT$^\omega$ by taking advantage of vacuous quantification.

\section{The sequent calculus $S_{\text{IK}^\omega}$}
The logic IK$^\omega$ is defined by \cite{Zardini2011-ZARTWC} for a first-order language based on the connectives $\rightarrow$, $\neg$, $\exists$ and $\forall$ in addition to $\otimes$ and $\oplus$ as symbols for conjunction and disjunction respectively. In this paper, we shall restrict our attention to a first-order language $\mathcal{L}$ based on the connectives $\exists$, $\rightarrow$ and $\neg$ where the set of well-formed formulas is defined with the standard recursive clauses. 

Following \cite{Zardini2011-ZARTWC}, IK$^\omega$ is defined with a sequent calculus based on sequents as multisets of $\mathcal{L}$-formulas which are such that both the antecedent and succedent multiset of a sequent can contain $\omega$ many formulas. Intuitively, multisets are collections of objects where their multiplicity matter but not their order. More formally, we will define a multiset of formulas as a pair $\langle X,f\rangle$ where $X$ is a set of formulas and $f$ a function from $X$ to $\omega+1$. In the presentation of the sequent calculus for IKT$^\omega$, we use the upper-case Latin letters $A$ and $B$ as meta-linguistic variables for formulas of $\mathcal{L}$, and upper-case Greek letters of the form $\Gamma$ and $\Delta$ as meta-linguistic variables for multisets of formulas. As usual, $\Gamma, A$ represents the multiset union of $\Gamma$ and $\{A\}$. $A[t/x]$ refers to the formula obtained by replacing every occurrence of $x$ in $A$ with $t$ and $\text{CTer}_\mathcal{L}$ is the set of closed terms of $\mathcal{L}$.

\begin{definition}
Let $S_{\text{IK}^\omega}$ be a sequent calculus with sequents as multisets of $\mathcal{L}$-formulas obtained with the initial sequents
\[
\Axiom$A,\Gamma\Rightarrow\fCenter\;\Delta,A$
\DisplayProof
\]
and the rules
\[
\Axiom$\Gamma\Rightarrow\fCenter\;\Delta,A$
\RightLabel{$\neg$L}
\UnaryInf$\neg A,\Gamma\Rightarrow\fCenter\;\Delta$
\DisplayProof
\quad
\Axiom$A,\Gamma\Rightarrow\fCenter\;\Delta$
\RightLabel{$\neg$R}
\UnaryInf$\Gamma\Rightarrow\fCenter\;\Delta,\neg A$
\DisplayProof
\]

\[
\Axiom$\Gamma\Rightarrow\fCenter\;\Delta,A$
\Axiom$B,\Gamma'\Rightarrow\fCenter\;\Delta'$
\RightLabel{$\rightarrow$L}
\BinaryInf$A\rightarrow B,\Gamma,\Gamma'\Rightarrow\fCenter\;\Delta,\Delta'$
\DisplayProof
\quad\quad
\Axiom$A,\Gamma\Rightarrow\fCenter\;\Delta,B$
\RightLabel{$\rightarrow$R}
\UnaryInf$\Gamma\Rightarrow\fCenter\;\Delta,A\rightarrow B$
\DisplayProof
\]

\[
\Axiom$A[t_0/x],\Gamma_0\Rightarrow\fCenter\;\Delta_0$
\Axiom$A[t_1/x],\Gamma_1\Rightarrow\fCenter\;\Delta_1$
\Axiom$A[t_2/x],\Gamma_2\Rightarrow\fCenter\;\Delta_2$
\Axiom$\ldots\fCenter$
\RightLabel{$\exists\text{L}^\omega$}
\QuaternaryInf$\exists xA,\Gamma_0,\Gamma_1,\Gamma_2,\ldots\Rightarrow\fCenter\;\Delta_0,\Delta_1,\Delta_2,\ldots$
\DisplayProof
\]

\[
\Axiom$\Gamma\Rightarrow\fCenter\;\Delta,A[t_0/x],A[t_1/x],A[t_2/x],\ldots$
\RightLabel{$\exists\text{R}^\omega$}
\UnaryInf$\Gamma\Rightarrow\fCenter\;\Delta,\exists xA$
\DisplayProof
\]

\noindent where $\{t_j\}_{j=0}^\infty$ represents a complete enumeration of $\text{CTer}_\mathcal{L}$. We refer to IK$^\omega$ as the logic defined as follows, where $\Gamma$ and $\Delta$ are multisets of formulas: 
\begin{equation*}
\text{$\langle\Gamma,\Delta\rangle\in\text{IK}^\omega$ if and only if $S_{\text{IK}^\omega}\vdash\Gamma\Rightarrow\Delta$.}
\end{equation*}
\end{definition}

\section{\L{}ukasiewicz-models}
\label{sec:lukasiewicz}
To simplify things, we  restrict our attention to valuations satisfying the conditions for first-order continuum-valued \L{}ukasiewicz-models, thus avoiding a domain of quantification. For how to  deal with a domain of quantification for continuum-valued \L{}ukasiewicz-models, we refer the interested reader to Hajek, Paris and Shepherdson \cite{Hajek2000-HAJTLP}.
\begin{definition}
A function $\mathcal{V}$ from the sentences of $\mathcal{L}$ to $[0,1]$ is a first-order continuum-valued \L{}ukasiewicz-valuation (\L{}Q-valuation) if and only if 
\begin{itemize}
    \item $\mathcal{V}(\neg A)=1-\mathcal{V}(A)$,
    \item $\mathcal{V}(A\rightarrow B)=\min\{1,1-\mathcal{V}(A)+\mathcal{V}(B)\}$,
    \item $\mathcal{V}(\exists xA)=\sup\{\mathcal{V}(A[t/x])\mid t\in\text{CTer}_\mathcal{L}\}$.
\end{itemize}
We refer to the logic defined as the set of formulas that is assigned the value $1$ in every \L{}Q-valuation as \L{}$_\infty$.
\end{definition}

It is quite natural to look at continuum-valued \L{}ukasiewicz-models for soundness in the case of $S_{\text{IK}^\omega}$ since propositional affine logic is a sublogic of propositional \L{}ukasiewicz-logic (see, e.g., Metcalfe, Olivetti and Gabbay \cite{ProofFuzzy}). However, $S_{\text{IK}^\omega}$ is not sound with regard to the above models. To show this, we require a way to interpret infinitary multiset-sequents in continuum-valued \L{}ukasiewicz-models. Following, for example, Gottwald \cite{sep-logic-manyvalued}, we can distinguish between strong and weak disjunction and conjunction in \L{}ukasiewicz-models based on the clauses used to define the connectives:
\begin{itemize}
    \item Weak disjunction: $\max\{\mathcal{V}(A),\mathcal{V}(B)\}$.
    \item Strong disjunction: $\min\{1,\mathcal{V}(A)+\mathcal{V}(B)\}$.
    \item Weak conjunction: $\min\{\mathcal{V}(A),\mathcal{V}(B)\}$.
    \item Strong conjunction: $\max\{0,\mathcal{V}(A)+\mathcal{V}(B)-1\}$.
\end{itemize} 
While we have not included connectives defined with these clauses in our language, we can nonetheless utilise the clauses for strong disjunction and conjunction to interpret $S_{\text{IK}^\omega}$-sequents in our models. In fact, we can interpret $S_{\text{IK}^\omega}$-sequents directly in the models by treating the antecedent as an infinitary strong conjunction and the succedent as an infinitary strong disjunction. This is in line with \cite{Zardini2011-ZARTWC}'s insistence on defining an infinitary multiset-multiset logic rather than a set-set logic, and it allows us to define an infinitary multiset-multiset logic on \L{}ukasiewicz-models.

\begin{definition}
A $S_{\text{IK}^\omega}$-sequent $\Gamma\Rightarrow\Delta$ is sound with regard to a \L{}Q-valuation $\mathcal{V}$ if and only if $$1-\min\Big\{1,\sum_{A\in\Gamma}(1-\mathcal{V}(A))\Big\}\leq\min\Big\{1,\sum_{B\in\Delta}(\mathcal{V}(B)\Big\}.$$
An $S_{\text{IK}^\omega}$-sequent $\Gamma\Rightarrow\Delta$ is valid if and only if $\Gamma\Rightarrow\Delta$ is sound with regard to every \L{}Q-valuation. An $S_{\text{IK}^\omega}$-rule $\mathcal{R}$ is sound with regard to \L{}Q-valuations if and only if, for every \L{}Q-valuation $\mathcal{V}$, if every premise $\Gamma'\Rightarrow\Delta'$ of $\mathcal{R}$ is sound with regard to $\mathcal{V}$, then the conclusion $\Gamma'\Rightarrow\Delta'$ of $\mathcal{R}$ is sound with regard to $\mathcal{V}$.
\end{definition}

In the following, we help ourselves to the notation $\mathcal{V}(\Gamma)$ and $\mathcal{V}(\Delta)$ where $\mathcal{V}(\Gamma)$ is an abbreviation of $1-\min\{1,\sum_{A\in\Gamma}(1-\mathcal{V}(A))\}$ and $\mathcal{V}(\Delta)$ is an abbreviation of $\min\{1,\sum_{B\in\Delta}\mathcal{V}(B)\}$. Importantly, $\Gamma$ and $\Delta$ are multisets, so the value of $A$/$B$ must be added once for each occurrence of formula $A$/$B$ in $\Gamma$/$\Delta$.
\begin{theorem}
The rule $\exists\text{R}^\omega$ of $S_{\text{IK}^\omega}$ is not sound with regard to \L{}Q-valuations.
\end{theorem}
\begin{proof}
Let $\mathcal{V}$ be a model such that $\mathcal{V}(\Gamma)=1$, $\mathcal{V}(\Delta)=0$ and $\mathcal{V}(A[t/x])=\frac{1}{2}$ for each $t\in\text{CTer}_\mathcal{L}$. It follows that $\mathcal{V}(\Gamma)\leq\min\{1,\sum^\infty_{i=0}(\mathcal{V}(A[t_i/x]))+\mathcal{V}(\Delta)\}$ where $(t_i)_{i=0}^\infty$ is a complete enumeration of $\text{CTer}_\mathcal{L}$. However, $\mathcal{V}(\Gamma)\leq\min\{1,\mathcal{V}(\exists xA)+\mathcal{V}(\Delta)\}$ does not hold since $\sup\{\mathcal{V}(A[t/x])\mid t\in\text{CTer}_\mathcal{L}\}<\sum^\infty_{i=0}(\mathcal{V}(A[t_i/x]))$.
\end{proof}

\section{Soundness for $S_{\text{IK}^\omega}$}
The sup-clause for $\exists$ is not strong enough for our purposes. This should not come as a surprise because the rules $\exists\text{L}^\omega$ and $\exists\text{R}^\omega$ are generalisations of strong disjunction, not weak disjunction. A natural proposal to overcome this problem is to replace the sup-clause for $\exists$ with the corresponding sum-clause as follows:
\begin{definition}
Let a first-order continuum-valued \L{}ukasiewicz-valuation with multiplicative quantifier (\L{}MQ-valuation) be like a \L{}Q-valuation with the exception that
$$\mathcal{V}(\exists xA)=\min\Big\{1,\sum^\infty_{i=0}\mathcal{V}(A[t_i/x])\Big\},$$ where $\{t_j\}_{j=0}^\infty$ is a complete enumeration of $t\in\text{CTer}_\mathcal{L}$.
\end{definition}
It remains to show that $S_{\text{IK}^\omega}$ is sound with regard to our new models where sequents are interpreted in the same way as above. To this end, we make the following definition. 

\begin{definition}
Let an $S_{\text{IK}^\omega}$-sequent $\Gamma\Rightarrow\Delta$ be sound with regard to an \L{}MQ-valuation $\mathcal{V}$ if and only if $$1-\min\Big\{1,\sum_{A\in\Gamma}(1-\mathcal{V}(A))\Big\}\leq\min\Big\{1,\sum_{B\in\Delta}\mathcal{V}(B)\Big\}.$$
An $S_{\text{IK}^\omega}$-sequent $\Gamma\Rightarrow\Delta$ is valid with regard to \L{}MQ-valuations if and only if $\Gamma\Rightarrow\Delta$ is sound with regard to every \L{}MQ-valuation. An $S_{\text{IK}^\omega}$-rule $\mathcal{R}$ is sound with regard to \L{}MQ-valuations if and only if, for every \L{}MQ-valuation $\mathcal{V}$, if every premise $\Gamma'\Rightarrow\Delta'$ of $\mathcal{R}$ is sound with regard to $\mathcal{V}$, then the conclusion $\Gamma'\Rightarrow\Delta'$ of $\mathcal{R}$ is sound with regard to $\mathcal{V}$.
\end{definition}
The following lemma  is essential for establishing that the rule $\exists\text{L}^\omega$ is sound.
\begin{lemma}
\label{lem:infinite}
Let $\Gamma = \{\gamma_i\}_{i=0}^\infty$, $\Delta = \{\delta_i\}_{i=0}^\infty$ and $X=\{\chi_i\}_{i=0}^\infty$ be infinite but enumerable sets of variables  $\gamma_i, \delta_i, \chi_i \in [0,1]$ so that for each index $i$, we have $$1-\min\big\{1,(1-\gamma_i)+(1-\chi_i)\big\}\leq\delta_i.$$
Then
$$1-\min\Big\{1,\sum^\infty_{i=0}(1-\gamma_i)+\big(1-\min\big\{1,\sum^\infty_{i=0}\chi_i\big\}\big)\Big\}\leq\min\Big\{1,\sum^\infty_{i=0}\delta_i\Big\}.$$
\end{lemma}
\begin{proof}
The assumptions imply the following (clearly equivalent) inequalities:
\begin{gather}
\tag{$*$}
    \chi_i\leq \delta_i + 1-\gamma_i,
\end{gather}
\begin{gather}
\tag{$**$}
    0\leq 1-\gamma_i-\chi_i+\delta_i.
\end{gather}
We now split the proof into two cases. 

\textbf{Case 1}: Suppose that $\sum^\infty_{i=0} \chi_i \geq 1$. The desired conclusion reduces to 
$$1\leq\min\Big\{1,\sum^\infty_{i=0}\delta_i\Big\}+\min\Big\{1,\sum^\infty_{i=0}(1-\gamma_i)\Big\}.$$
Now, if either $\sum^\infty_{i=0}\delta_i\geq 1$ or $\sum^\infty_{i=0}(1-\gamma_i)\geq 1$ hold, then the  conclusion  follows trivially. If not, then the desired inequality can be expressed on the form
\begin{align*}
    1 &\leq 
     \sum^\infty_{i=0}(\delta_i + 1-\gamma_i),
\end{align*}
where we used that, by our assumptions, all involved series converge. Clearly, this   inequality  follows immediately from ($\ast$).

\textbf{Case 2}: Suppose that $\sum^\infty_{i=0}\chi_i<1$. The desired conclusion now becomes
$$1-\min\Big\{1,\sum^\infty_{i=0}(1-\gamma_i)+(1-\sum^\infty_{i=0}\chi_i)\Big\}\leq\min\Big\{1,\sum^\infty_{i=0}\delta_i\Big\},$$
which we rewrite on the form
$$1\leq\min\Big\{1,1+\sum^\infty_{i=0}(1-\gamma_i)-\sum^\infty_{i=0}\chi_i\Big\}+\min\Big\{1,\sum^\infty_{i=0}\delta_i\Big\}.$$
This is trivially true if $1+\sum^\infty_{i=0}(1-\gamma_i) - \sum_{i=0}^\infty \chi_i \geq 1$ or $\sum^\infty_{i=0}\delta_i\geq 1$, so we consider the case in which both are strictly less than $1$. As above, since all involved series are convergent, we can rewrite the desired conclusion on the form 
\begin{align*}
    1 
    \leq 1+\sum^\infty_{i=0}(1-\gamma_i - \chi_i + \delta_i).
\end{align*}
Clearly, this   inequality follows immediately from    $(**)$.
\end{proof}
With this lemma at hand, soundness is straight-forward.
\begin{theorem}
\label{thm:consistency}
If there is a derivation of $\Gamma\Rightarrow\Delta$ in $S_{\text{IK}^\omega}$, then $\Gamma\Rightarrow\Delta$ is valid with regard to \L{}MQ-valuations.
\end{theorem}
\begin{proof}
The proof proceeds by induction on the construction of a derivation for $\Gamma\Rightarrow\Delta$. We present the salient cases for illustration.

\textit{Initial sequents}: With $\mathcal{V}(\Gamma)$ and $\mathcal{V}(\Delta)$ being in $[0,1]$, we observe that $\mathcal{V}(\Gamma)-1$ $\leq\mathcal{V}(\Delta)$ and thus that $\mathcal{V}(A)+\mathcal{V}(\Gamma)-1\leq\mathcal{V}(A)+\mathcal{V}(\Delta)$. This implies $$1-\big(1-\mathcal{V}(A)+1-\mathcal{V}(\Gamma)\big)\leq\mathcal{V}(A)+\mathcal{V}(\Delta),$$ which in turn implies $$1-\min\big\{1,1-\mathcal{V}(A)+1-\mathcal{V}(\Gamma)\big\}\leq\mathcal{V}(A)+\mathcal{V}(\Delta).$$ With $1-\min\{1,1-\mathcal{V}(A)+1-\mathcal{V}(\Gamma)\}\leq 1$, it follows finally that $$1-\min\big\{1,1-\mathcal{V}(A)+1-\mathcal{V}(\Gamma)\big\}\leq\min\big\{1,\mathcal{V}(A)+\mathcal{V}(\Delta)\big\}.$$

$\exists\text{R}^\omega$: Assume that $\mathcal{V}(\Gamma)\leq\min\{1,\sum^\infty_{i=0}\mathcal{V}(A[t_i/x])+\mathcal{V}(\Delta)\}$. We now consider two cases. First, we suppose that $\sum^\infty_{i=0}\mathcal{V}(A[t_i/x])<1$. Then $\sum^\infty_{i=0}\mathcal{V}(A[t_i/x])$ is equal to $\min\{1,\sum^\infty_{i=0}\mathcal{V}(A[t_i/x])\}$, and we are done. In the second case,  we suppose that $\sum^\infty_{i=0}\mathcal{V}(A[t_i/x])\geq 1$. Under this assumption, the expression $$\min\Big\{1,\sum^\infty_{i=0}\mathcal{V}(A[t_i/x])+\mathcal{V}(\Delta)\Big\}$$ is equal to $$\min\Big\{1,\min\big\{1,\sum^\infty_{i=0}\mathcal{V}(A[t_i/x])\big\}+\mathcal{V}(\Delta)\Big\}.$$

$\exists\text{L}^\omega$: Assume that, for each closed term $t_i$, we have a triple $\langle\Gamma_i,At_i,\Delta_i\rangle$ such that $1-\min\{1,(1-\mathcal{V}(\Gamma_i))+(1-\mathcal{V}(At_i))\}\leq\mathcal{V}(\Delta_i)$. By Lemma \ref{lem:infinite}, we immediately obtain 
$$1-\min\Big\{1,\sum^\infty_{i=0}(1-\mathcal{V}(\Gamma_i))+(1-\mathcal{V}(\exists xA)\Big\}\leq\min\Big\{1,\sum^\infty_{i=0}\mathcal{V}(\Delta_i)\Big\}$$
through the clause for $\exists$.
\end{proof}

\section{Adding Transparent Truth?}
In \cite{Zardini2011-ZARTWC}, the language $\mathcal{L}$ is expanded with a designated one-place predicate $T$ together with a denumerable set of constants ``to serve as canonical names of all sentences in the language'' \cite[pp. 506]{Zardini2011-ZARTWC}. The constant for a formula $A$ will be referred to by either $\ulcorner A\urcorner$ or some lower-case letter as stipulated. Moreover, the calculus for IK$^\omega$ is expanded with the following rules to define the calculus for IKT$^\omega$ which we here label as $S_{\text{IKT}^\omega}$:
\[
\Axiom$A,\Gamma\Rightarrow\fCenter\;\Delta$
\RightLabel{TL}
\UnaryInf$T\ulcorner A\urcorner,\Gamma\Rightarrow\fCenter\;\Delta$
\DisplayProof
\quad
\Axiom$\Gamma\Rightarrow\fCenter\;\Delta,A$
\RightLabel{TR}
\UnaryInf$\Gamma\Rightarrow\fCenter\;\Delta,T\ulcorner A\urcorner$
\DisplayProof
\]
The predicate $T$ is thus intended as a transparent truth predicate. A cut-elimination proof for $S_{\text{IKT}^\omega}$ is presented by \cite{Zardini2011-ZARTWC} from which consistency of IKT$^\omega$ is concluded.

Now, it has been shown by Da Ré and Rosenblatt \cite{DaRe2018-DARCIQ} that if the simple coding scheme for generating self-reference used by \cite{Zardini2011-ZARTWC} to define $S_{\text{IKT}^\omega}$ is replaced with a G\"{o}del coding based on a theory of arithmetic containing function symbols for certain primitive recursive functions, then the resulting theory is not only $\omega$-inconsistent but also inconsistent.

The proof presented by \cite{DaRe2018-DARCIQ} relies on a result by Bacon \cite{Bacon2013-BACCPA-3} who showed that a transparent theory of theory based on Peano arithmetic is $\omega$-inconsistent if it satisfies the following additional conditions:
\begin{itemize}
    \item The language contains a one-place function symbol $\myd{\rightarrow}$ satisfying the equation $\ulcorner A\urcorner\myd{\rightarrow}\ulcorner B\urcorner=\ulcorner A\rightarrow B\urcorner$, and a two-place function symbol $\myd{f_b}$ satisfying the equations $\myd{f_b}(0,x)=x\myd{\rightarrow}\ulcorner\bot\urcorner$ and $\myd{f_b}(n+1,x)=x\myd{\rightarrow}\myd{f_b}(n,x)$ where $+$ is addition.
    \item The theory satisfies the following principles, here presented in sequent calculus format for uniformity:
\[
\Axiom$A\Rightarrow\fCenter\;B$
\RightLabel{B1}
\UnaryInf$\exists xA\Rightarrow\fCenter\;\exists xB$
\DisplayProof
\quad\quad
\Axiom$\fCenter$
\RightLabel{B2}
\UnaryInf$A\rightarrow\exists xB\Rightarrow\fCenter\;\exists x(A\rightarrow B)$
\DisplayProof
\]
\end{itemize}
\noindent Relying on the weak diagonal lemma, the proof of $\omega$-inconsistency in \cite{Bacon2013-BACCPA-3} employs the instance $\mu\leftrightarrow\exists xT\myd{f_b}(x,\ulcorner\mu\urcorner)$. Importantly, the theory of truth obtained by expanding \L{}$_\infty$ with sufficient arithmetic to define the above function symbols and transparent truth based on a suitable G\"{o}del coding satisfies these conditions. This result can be seen as a refinement of the observations by Restall \cite{Restall1992-RESAAT-3} and \cite{Hajek2000-HAJTLP} in which different primitive recursive functions are used for the same purpose with regard to \L{}$_\infty$, namely to show that the resulting theory of truth is $\omega$-inconsistent.

In \cite{DaRe2018-DARCIQ}, it is observed that B1 and B2 also hold in IKT$^\omega$, and it is shown that not only does the $\omega$-inconsistency result in \cite{Bacon2013-BACCPA-3} hold in a suitable variant of $S_{\text{IKT}^\omega}$, but the resulting theory is outright inconsistent through the rule $\exists\text{L}^\omega$. 

The result by \cite{DaRe2018-DARCIQ} can be strengthened in the following way. The formula used in \cite{Bacon2013-BACCPA-3} involves the truth predicate, the existential quantifier, the conditional and $\bot$ representing absurdity. As it turns out, we can, with inspiration from the result in Fjellstad \cite{Fjellstad2018-FJEICR}, replace the implication to $\bot$ in \cite{Bacon2013-BACCPA-3}'s formula with negation and thus employ a variation of the formula utilised by McGee \cite{McGee1985-MCGHTC} for a $\omega$-inconsistency result concerning some classical theories of truth.

\begin{proposition}
Let $S'$ be the sequent calculus obtained by expanding $S_{\text{IKT}^\omega}$ where $\ulcorner A\urcorner$ is the numeral of the G\"{o}del-code for a formula $A$ with a theory of arithmetic satisfying the following conditions: 
\begin{itemize}
    \item It includes the following equations for a one-place function symbol $\myd{T}$ and a two-place function symbol $\myd{f}_m$:
$$\myd{T}t=\ulcorner Tt\urcorner\quad\quad\myd{f}_m(0,\ulcorner A\urcorner)=\ulcorner A\urcorner\quad\quad\myd{f}_m(n+1,\ulcorner A\urcorner)=\myd{T}\myd{f}_m(n,\ulcorner A\urcorner)$$
    \item For each closed term $t$ there is a numeral $n$ such that $\Rightarrow t=n$ is derivable.
    \item If $\Rightarrow s=t$ is derivable then $A$ and $A[s/t]$ are intersubstitutable in every derivable sequent.
    \item There is a formula $\mu$ such that $\mu$ and $\neg\exists xT\myd{f}_m(x,\ulcorner\mu\urcorner)$ are intersubstitutable in every derivable sequent.
\end{itemize}
Then the theory defined with $S'$ is $\omega$-inconsistent and inconsistent.
\end{proposition}
\begin{proof}
The following derivation is now available where some steps are left implicit for readability:

\begin{small}
\[
\Axiom$T\myd{f}_m(0,\ulcorner\mu\urcorner)\Rightarrow\fCenter\;T\myd{f}_m(0,\ulcorner\mu\urcorner)$
\RightLabel{$T$R}
\UnaryInf$T\myd{f}_m(0,\ulcorner\mu\urcorner)\Rightarrow\fCenter\;T\ulcorner T\myd{f}_m(0,\ulcorner\mu\urcorner)\urcorner$
\RightLabel{def.$\myd{f}_m$}
\UnaryInf$T\myd{f}_m(0,\ulcorner\mu\urcorner)\Rightarrow\fCenter\;T\myd{f}_m(1,\ulcorner\mu\urcorner)$
\Axiom$T\myd{f}_m(1,\ulcorner\mu\urcorner)\Rightarrow\fCenter\;T\myd{f}_m(2,\ulcorner\mu\urcorner)$
\Axiom$\ldots\fCenter$
\RightLabel{$\exists\text{L}^\omega$}
\insertBetweenHyps{\hskip 3pt}
\TrinaryInf$\exists xT\myd{f}_m(x,\ulcorner\mu\urcorner)\Rightarrow\fCenter\;T\myd{f}_m(1,\ulcorner\mu\urcorner),T\myd{f}_m(2,\ulcorner\mu\urcorner),\ldots$
\RightLabel{$\neg$R}
\UnaryInf$\Rightarrow\fCenter\;\neg\exists xT\myd{f}_m(x,\ulcorner\mu\urcorner),T\myd{f}_m(1,\ulcorner\mu\urcorner),T\myd{f}_m(2,\ulcorner\mu\urcorner),\ldots$
\RightLabel{$T$R}
\UnaryInf$\Rightarrow\fCenter\;T\myd{f}_m(0,\ulcorner\mu\urcorner),T\myd{f}_m(1,\ulcorner\mu\urcorner),T\myd{f}_m(2,\ulcorner\mu\urcorner),\ldots$
\RightLabel{$\exists\text{R}^\omega$}
\UnaryInf$\Rightarrow\fCenter\;\exists xT\myd{f}_m(x,\ulcorner\mu\urcorner)$
\RightLabel{$\neg$L}
\UnaryInf$\neg\exists xT\myd{f}_m(x,\ulcorner\mu\urcorner)\Rightarrow\fCenter$
\RightLabel{$T$L/def.$\myd{f}_m$}
\UnaryInf$T\myd{f}_m(0,\ulcorner\mu\urcorner)\Rightarrow\fCenter$
\DisplayProof
\]
\end{small}

Through iterated applications of the definition of $\myd{f}_m$ and $T$L we obtain the sequent $T\myd{f}_m(i,\ulcorner\mu\urcorner)\Rightarrow$ for each numeral $i$ and thus for each closed term. The defined theory is therefore $\omega$-inconsistent. With one application of $\exists\text{L}^\omega$ and one application of $\neg$R, we finally obtain the sequent
$\Rightarrow\neg\exists xT\myd{f}_m(x,\ulcorner\mu\urcorner)$. The theory is inconsistent.\end{proof}

This result shows that while the conditional is required for $\omega$-inconsistency in the case of \L{}$_\infty$ with transparent truth, this is not the case with IKT$^\omega$.

This result is clearly not ideal for the project initiated by \cite{Zardini2011-ZARTWC} where $S_{\text{IK}^\omega}$ is presented for the purpose of being a logic for reasoning about transparent truth as it brings forth a significant limitation with the approach. However, one might nonetheless think that the issue is not worse for $S_{\text{IKT}^\omega}$ than for \L{}$_\infty$ with transparent truth or the classical theory of truth FS presented by Friedman and Sheard \cite{Friedman1987-FRIAAA} and further explored by Halbach \cite{Halbach1994-VOLASO-2} which satisfies the conditions for \cite{McGee1985-MCGHTC}'s $\omega$-inconsistency result. For example, one could perhaps block the inconsistency by supplying the theory with numerals for non-standard numbers, thereby avoiding that $\exists\text{L}^\omega$ is turned into a Hilbertian $\omega$-rule. Moreover, since \L{}$_\infty$ can be expanded with transparent truth, as shown in \cite{Hajek2000-HAJTLP}, as long as certain primitive recursive functions are not captured by some function symbols, it might even be tempting to investigate whether the soundness proof from the previous section can be utilised as a consistency proof of IKT$^\omega$ under similar constraints. 

Indeed, despite their negative result, Da Ré and Rosenblatt \cite{DaRe2018-DARCIQ} still express confidence in that the original theory IKT$^\omega$ is consistent with reference to the cut-elimination proof presented by \cite{Zardini2011-ZARTWC}. However, that proof has since   been shown  in \cite{Fjellstad2020-FJEANO} to contain an error. In particular, it is shown that the reduction step employed for the quantifier rules is inadequate. It follows that the claim by \cite{Zardini2011-ZARTWC} that IKT$^\omega$ is consistent is unsupported. This would thus be another reason for investigating whether the soundness proof can be turned into a consistency proof.

This is not the case. With inspiration from the discussion about quantifiers in Rosenblatt \cite{Rosenblatt2019-ROSNCL}, we can show the following:
\begin{proposition}
\L{}MQ-valuations are inconsistent with the condition that the language contains a one-place predicate $T$ and a constant $l$ such that the following holds.
$$\mathcal{V}(Tl)=\mathcal{V}(\neg\exists xTl).$$
\end{proposition}
\begin{proof}
Assume that $\mathcal{V}(Tl)>0$. It follows that $\mathcal{V}(\exists xTl)=1$ since, for each closed term $t$, it holds that $\mathcal{V}(Tl[t/x])=\mathcal{V}(Tl)$. This implies that $\mathcal{V}(\neg\exists xTl)=0$ and thus $\mathcal{V}(Tl)=0$, which contradicts our assumption. Assume instead that $\mathcal{V}(Tl)=0$. It follows that $\mathcal{V}(\exists xTl)=0$ by the $\exists$-clause and, moreover, that $\mathcal{V}(\neg\exists xTl)=0$ by the condition for $Tl$. This contradicts the $\neg$-clause.
\end{proof}
Note the vacuous quantification in the formula $\neg\exists xTl$; while dubious it is certainly permissible as per the language stipulations in \cite{Zardini2011-ZARTWC}. In fact, we can show the corresponding result for IKT$^\omega$.

\begin{proposition}
If the language for $S_{\text{IKT}^\omega}$ contains a constant $l$ such that $l$ serves as canonical name for the formula $\neg\exists xTl$ then IKT$^\omega$ is inconsistent.
\end{proposition}
\begin{proof}
Through vacuous quantification we have the following derivation:
\[
\Axiom$Tl\Rightarrow\fCenter\;Tl$
\Axiom$Tl\Rightarrow\fCenter\;Tl$
\Axiom$\ldots\fCenter$
\RightLabel{$\exists\text{L}^\omega$}
\TrinaryInf$\exists xTl\Rightarrow\fCenter\;Tl,Tl,\ldots$
\RightLabel{$\neg$R}
\UnaryInf$\Rightarrow\fCenter\;\neg\exists xTl,Tl,Tl,\ldots$
\RightLabel{TR}
\UnaryInf$\Rightarrow\fCenter\;Tl,Tl,Tl,\ldots$
\RightLabel{$\exists\text{R}^\omega$}
\UnaryInf$\Rightarrow\fCenter\;\exists xTl$
\RightLabel{$\neg$L}
\UnaryInf$\neg\exists x Tl\Rightarrow\fCenter$
\RightLabel{TL}
\UnaryInf$Tl\Rightarrow\fCenter$
\Axiom$Tl\Rightarrow\fCenter$
\Axiom$\ldots\fCenter$
\RightLabel{$\exists\text{L}^\omega$}
\TrinaryInf$\exists xTl\Rightarrow\fCenter$
\RightLabel{$\neg$R}
\UnaryInf$\Rightarrow\fCenter\;\neg\exists xTl$
\DisplayProof
\]
\end{proof}
Adding a copy of $Tl$ to an infinite sequence of $Tl$'s makes no difference to the totality of copies.

The issue of vacuous quantification was not addressed in \cite{Zardini2011-ZARTWC}, and we have thus taken for granted that it should be treated in the way that supports the above results. We maintain that this is reasonable, despite there being two options for how to understand vacuous quantification in the current setting. The first option is to require an infinite sequence of $Tl$ as above:
\[
\Axiom$\Gamma\Rightarrow\fCenter\;\Delta,Tl[t_0/x],Tl[t_1/x],Tl[t_2/x],\ldots$
\RightLabel{$\exists\text{R}^\omega$}
\UnaryInf$\Gamma\Rightarrow\fCenter\;\Delta,\exists xTl$
\DisplayProof
\]
Let us call this ``the multiplicative way''. The second option is to require that the rule takes only one copy of that formula as premise:
\[
\Axiom$\Gamma\Rightarrow\fCenter\;\Delta,Tl$
\RightLabel{$\exists\text{R}^\omega$}
\UnaryInf$\Gamma\Rightarrow\fCenter\;\Delta,\exists xTl$
\DisplayProof
\]
The latter option, which we can call ``the additive way'', comes with a contractive flavour and seems thus not to be in tune with the spirit of \cite{Zardini2011-ZARTWC}. After all, the corresponding restriction on $\exists\text{L}^\omega$ permits the derivation of $\exists xTl\Rightarrow Tl$ which actually amounts to the claim that $\bigoplus Tl$ implies $Tl$ if we consider vacuous existential quantification as corresponding to the infinitary multiplicative disjunction of a formula. The resulting logic would thus be such that the former inference is valid while the inference from $Tl\oplus Tl$ to $Tl$ is invalid.

\section{Conclusion}
We have shown that the logic $S_{\text{IK}^\omega}$ is sound with regard to \L{}ukasiewicz valuations that are modified by treating the existential quantifier as a generalisation of the strong disjunction instead of the weak disjunction. However, we have also shown that this soundness proof cannot be used as consistency proof for the theory of truth IKT$^\omega$ as long as vacuous quantification is treated in the multiplicative way. The problem is that valuations are in that case inconsistent with transparent truth. In addition, IKT$^\omega$ is inconsistent under the same condition.

\section*{Acknowledgements}

The first author would like to acknowledge the utility of the discussions with Casper Storm Hansen, Fausto Barbero and Uwe Petersen for the preparation of the material for this paper. The first author's research was supported by the Research Council of Norway under grant number 262837/F10 (co-funded by the European Union’s Seventh Framework Programme for research, technological development and demonstration under Marie Curie grant agreement no 608695).


%
\end{document}